\documentclass[12pt,a4paper,leqno]{amsart}
\usepackage{amssymb,amsmath}
\usepackage{xcolor}
\usepackage{bbm}
\usepackage{bigints}
\usepackage{mathtools} 
\usepackage{mathrsfs} 
\usepackage{hyperref} 
\usepackage{bigints}

\numberwithin{equation}{section}
\input colordvi
\newcommand\supp{\mathrm{supp}}

\allowdisplaybreaks

\advance\textwidth by 2.5cm
\advance\oddsidemargin by -1.6cm
\advance\evensidemargin by -1.6cm

\newtheorem{Theorem}{Theorem}[section]
\newtheorem{Definition}[Theorem]{Definition}
\newtheorem{Proposition}[Theorem]{Proposition}

\newcommand\restr[2]{{
		\left.\kern-\nulldelimiterspace 
		#1 
		\vphantom{\big|} 
		\right|_{#2} 
}}

\begin{document}

\bibliographystyle{plain}
\pagenumbering{arabic}

\title[SGWE for rough data]{Low regularity solutions to the stochastic geometric wave equation driven by a fractional Brownian sheet}
\author{Zdzis{\l}aw Brze\'zniak and Nimit Rana}
\address{Department of Mathematics \\
The University of York \\
Heslington, York,  UK} \email{zdzislaw.brzezniak@york.ac.uk}
\address{Department of Mathematics \\
	The University of York \\
	Heslington, York,  UK} \email{nr734@york.ac.uk}
\today

\addtocounter{footnote}{-1} \vskip 1 true cm

\begin{abstract}
  We announce a result on the existence of a unique local solution to a stochastic geometric wave equation on the one dimensional Minkowski space $\mathbb{R}^{1+1}$ with values in  an arbitrary compact Riemannian manifold. We consider a rough initial data in the sense that its regularity is lower than the energy critical.
\end{abstract}

\maketitle

\section{Introduction} \label{sec:Intro}
Recently, the existence and the uniqueness of a global solution, in the strong and weak sense, for the stochastic geometric wave equations (SGWEs) on the Minkowski space $\mathbb{R}^{1+m}$, $m \geq 1$, with  the target manifold $(N,g)$ being  a suitable $n$-dimensional Riemannian manifold,  e.g. a sphere, has been established under  various sets of assumptions by the first named author and M. Ondrej\'at, see  \cite{BO2007,BO2011,BO2013} for details. To the best of  our knowledge, the most general result in the case  $m=1$,  is a   construction of  a global $H_{loc}^1(N) \times L_{loc}^2 (TN)$-valued weakly continuous solution of SGWE, where $TN$ denotes the tangent bundle of $N$,  see \cite{BO2011}.

The purpose of this note is to present a method by which we can prove the existence of a unique local solution to SGWE with  $m=1$ in the case of the initial data belonging to $H_{loc}^s(N) \times H_{loc}^{s-1}(TN)$ for $s \in \left( \frac{3}{4},1\right)$. In particular, we generalize the corresponding deterministic theory result of \cite{KT1998} to the stochastic setting, as well as the results of \cite{BO2007}-\cite{BO2013} to the wave maps equation with low regularity initial data (i.e. $s<1$) and fractional (both in time and space) Gaussian noise.

A more detailed account of this work and the global theory, with complete proofs, will be presented  in forthcoming  papers.

\section{Problem formulation}
We are interested in  solutions  having continuous paths and hence,  motivated by \cite{KT1998} and \cite{QT2007}, we find that it is suitable to formulate the Cauchy problem for the SGWE  using  local coordinates on the target manifold $N$. To be precise, given a  sufficiently smooth function $\sigma$  from $\mathbb{R}^n$ to $\mathbb{R}^n$, for the wave map $z: \mathbb{R}^{1+1} \to N$ composed with a given local chart $\phi$ of $N$ we consider the following Cauchy problem
\begin{align}\label{introSGWE}
& \Box u = N_0(u) + \sigma(u) \dot{\xi} , \nonumber\\
& u(0,x)  = u_0(x), \quad \textrm{ and } \quad \partial_{t} u(0,x) = u_1(x),
\end{align}
where $\phi \circ z:= u: \mathbb{R}^{1+1} \to \mathbb{R}^n$; $\Box := \partial_{t}^2 - \Delta_x$;
$$ \partial_0 = \partial_t, \quad \partial^0 = -\partial_t, \quad \partial_1 = \partial^1 = \partial_x ;$$
 $N_0(u) := - \sum_{a,b=1}^{n} \sum_{\mu=0}^{1}\Gamma_{ab}^k(u) \left(  \partial_{\mu} u^a \partial^{\mu} u^b \right)$ with  $\Gamma_{ab}^k$ denoting  the Christoffel symbols on $N$ in the chosen  local  coordinate system  and  $\xi$ is a suitable random field. The necessary assumptions will be  given later in a precise manner.

An efficient way to simplify the computations  of the required a'priori estimates for \eqref{introSGWE} is to switch the coordinate-axis of $(t,x)$-variables to the null coordinates, see for instance \cite{KT1998,MNT2010}, and respectively \cite{Walsh1986B}, for the deterministic and the stochastic problem. Our approach is in line with these references.  By performing the following transformation, which can be made rigorous for sufficiently regular case,
\begin{equation}\label{CoordTranformation}
	u^*(\alpha,\beta) := u\left(\frac{\alpha + \beta}{2}, \frac{\alpha - \beta}{2} \right) = u(t,x) \textrm{ and } u(t,x) = u^*(t+x,t-x),
\end{equation} %
the problem \eqref{introSGWE} can be re-written as
\begin{align}\label{SGWEab}
& \Diamond u^* = \mathcal{N}(u^*) + \sigma(u^*) \varXi_{\alpha\beta},
\end{align}
where $\varXi_{\alpha\beta} :=\frac{\partial^2\Xi}{\partial \alpha\partial \beta}$,  subject to the following boundary conditions
\begin{align}\label{SGWEabCons}
&  u^*(\alpha,-\alpha) = u_0(\alpha), \;\;\;  \partial_{\alpha}u^*(\alpha,-\alpha) + \partial_{\beta}u^*(\alpha,-\alpha) = u_1(\alpha).
\end{align}
Here $\Xi$  is a fractional Brownian sheet (fBs) on $\mathbb{R}^2$ with Hurst indices $H_1, H_2 \in (0,1)$, i.e.  $\Xi$ is a centered Gaussian process such that
\begin{equation}\label{fBs}
\mathbb{E}\big[\Xi(\alpha_1, \beta_1) \Xi(\alpha_2, \beta_2)\bigr] = R_{H_1}(\vert \alpha_1 \vert,\vert \alpha_2 \vert) R_{H_2}(\vert \beta_1 \vert,\vert \beta_2 \vert), \quad  (\alpha_1,\beta_1), (\alpha_2,\beta_2) \in \mathbb{R}^2, \nonumber
\end{equation}
where  $R_H (a,b) = \frac{1}{2} \left(a^{2H} + b^{2H} - \vert a-b \vert^{2H} \right)$,   $ a,b \in \mathbb{R}$ and
\begin{align*}
& \Diamond u^* :=  4 \frac{\partial^2 u^*}{\partial \alpha \partial \beta}, \;\;\;  \mathcal{N}(u^*) :=  4 \sum_{a,b=1}^{n} \Gamma_{ab}(u^*) \frac{\partial u^{*a}}{\partial \alpha}  \frac{\partial u^{*b}}{\partial \beta} . \nonumber
\end{align*}

From now on we will  only work in the $(\alpha,\beta)$-coordinates  and hence,  we will write $u$ instead of $u^*$ in the sequel. As usual in the SPDE theory, we understand the SGWE \eqref{SGWEab} in the following integral/mild form
\begin{align}\label{eqn-SGWEIntegralForm}
u = S(u_0,u_1) + \Diamond^{-1} \mathcal{N}(u) + \Diamond^{-1}\bigl[ \sigma(u) \varXi_{\alpha\beta} \bigr],
\end{align}
where, for $(\alpha,\beta)\in \mathbb{R}^2$,  \begin{equation}\label{SGWEIntegralFormHomgeneous}
[S(u_0,u_1)](\alpha,\beta) :=  \frac{1}{2} \left[ u_0(\alpha) + u_0(-\beta) \right] + \frac{1}{2} \int_{-\beta}^{\alpha} u_1(r) \,dr;
\end{equation}
\begin{equation}\label{SGWEIntegralFormInHomgeneous}
\left[ \Diamond^{-1} \mathcal{N}(u) \right](\alpha,\beta) := \frac{1}{4} \int_{-\beta}^{\alpha} \int_{-a}^{\beta} \mathcal{N}(u(a,b)) \,db\,da ;
\end{equation}
and
\begin{equation}\label{SGWEIntegralFormNoise}
\left[ \Diamond^{-1} (\sigma(u) \varXi_{\alpha\beta}) \right](\alpha,\beta) := \frac{1}{4}\int_{-\beta}^{\alpha} \int_{-a}^{\beta} \sigma(u(a,b))~ \varXi_{\alpha\beta}(da,db).
\end{equation}
The integral on the right hand side  of  \eqref{SGWEIntegralFormNoise} is well-defined pathwise, see   Prop. \ref{StocIntWelDef}.

\section{Relevant notation and function spaces}
If $x$ and $y$ are two quantities (typically non-negative), we will write $x \lesssim y$ or $y \gtrsim x$ to denote the statement that $x \leq Cy$ for some positive constant $C > 0$.

By $L^p(\mathbb{R}^d)$, for $p \in [1,\infty)$, we denote the classical real Banach space of all (equivalence classes of) $\mathbb{R}$-valued $p$-integrable functions on $\mathbb{R}^d$.
For $s \in \mathbb{R}$, we set \begin{equation}\nonumber
H^s(\mathbb{R}^d) = \{f \in \mathcal{S}^\prime(\mathbb{R}^d) : \|f\|_{H^s(\mathbb{R}^d)} := \int_{\mathbb{R}^d} \langle \xi \rangle^{2s} ~| [\mathcal{F}(f)](\xi)|^2 \, d\xi < \infty  \},
\end{equation}
where $\mathcal{S}^\prime(\mathbb{R}^d)$ is the set of all tempered distributions on $\mathbb{R}^d$, i.e. the dual of    the Schwartz space $\mathcal{S}(\mathbb{R}^d)$ of all rapidly decreasing infinitely differentiable functions on $\mathbb{R}^d$,  and  $\langle \xi \rangle: = (1+\vert \xi \vert^2)^{1/2}$, $\xi \in \mathbb{R}^d$ and $\mathcal{F}(f)$ is the $d$-dimensional Fourier transform of $f$.

\begin{Definition}\label{Hsdelta}
	Let $s,\delta \in \mathbb{R}$. The hyperbolic $H^{s,\delta}$ and the product $H_t^sH_x^\delta $  Sobolev spaces are
 the sets of all $u \in \mathcal{S}^\prime(\mathbb{R}^2)$ for which, the appropriate norm is finite, where, with  $\mathcal{F}(u)$	being the space-time Fourier transform of $u  \in \mathcal{S}^\prime(\mathbb{R}^2)$,
	\begin{eqnarray*}
	\Vert u \Vert_{H^{s,\delta}} &: =&  \left( \int_{\mathbb{R}^2}  \langle \vert \tau \vert + \vert \xi\vert \rangle^{2s} \langle \vert \tau \vert - \vert \xi\vert \rangle^{2\delta}  \vert [\mathcal{F}(u)](\tau,\xi)\vert^2  \, d\xi \, d\tau \right)^{1/2},
\\	
\Vert u \Vert_{H_t^s H_x^\delta} &:= & \left(\int_{\mathbb{R}^2} \langle \tau \rangle^{2s}  \langle \xi \rangle^{2\delta} \vert [\mathcal{F}(u)](\tau,\xi) \vert^2 \, d\tau \, d \xi \right)^{1/2}.
\end{eqnarray*}

\end{Definition}
Let $\Phi(\mathbb{R}^d)$ be the set of all systems $\varphi=\{ \varphi_j \}_{j=0}^\infty \subset \mathcal{S}(\mathbb{R}^d)$ such that
\begin{enumerate}
	\item  $\textrm{supp}~\varphi_0 \subset \{ x: \vert x \vert \leq 2 \} $,
		$\textrm{supp}~\varphi_j \subset \{ x: 2^{j-1} \leq \vert x \vert \leq 2^{j+1} \}, ~ \textrm{ if } j \in \mathbb{N}\setminus\{0\}$.\\
	\item For every multi-index $\alpha$ there exists a positive number $C_\alpha$ such that \begin{equation}\nonumber
		2^{j \vert \alpha \vert} D^\alpha \varphi_j(x) \leq c_\alpha \textrm{ for all } j \in \mathbb{N}   \textrm{ and all } x \in \mathbb{R}^d.
	\end{equation}
	\item  $\sum\limits_{j=0}^{\infty} \varphi_j(x)=1 \textrm{ for every } x \in \mathbb{R}^d.$
\end{enumerate}

It is known, see  \cite[Remark 2.3.1/1]{Triebel1983B}, that the system $\Phi(\mathbb{R}^d)$ is not empty. Given a dyadic partition of unity $\varphi := \{ \varphi_j \}_{j=0}^\infty \in \Phi(\mathbb{R})$ and  a tempered distribution  $f \in \mathcal{S}^\prime(\mathbb{R}^2)$, the Littlewood-Paley blocks of $f$ are defined as $\Delta_{j,k} f:= 0$,    $j,k \leq -1$, and
\begin{align}\nonumber
& \Delta_{j,k}f := \mathcal{F}^{-1} (\varphi_j(\tau) \varphi_k(\xi)[\mathcal{F}(f)](\tau,\xi)), ~~  j,k \geq 0,
\end{align}
where $\mathcal{F}^{-1}$ stands for the inverse Fourier transform on $\mathcal{S}^\prime(\mathbb{R}^2)$.
Next,  for $(s_1,s_2) \in \mathbb{R}^2$, $p,q \in (1,\infty)$, we define the following Banach space \begin{equation}\nonumber
	S_{p,q}^{s_1,s_2}B(\mathbb{R}^2) = \{ f \in \mathcal{S}^\prime(\mathbb{R}^2): \Vert f \Vert_{S_{p,q}^{s_1,s_2}B(\mathbb{R}^2)}^\varphi  < \infty  \},
\end{equation} where
\begin{equation}\nonumber
\Vert f \Vert_{S_{p,q}^{s_1,s_2}B(\mathbb{R}^2)}^\varphi := \left(\sum\limits_{k=0}^{\infty}\sum\limits_{j=0}^{\infty}2^{q(s_1j+s_2k)} \Vert \Delta_{j,k}f \Vert_{L^p(\mathbb{R}^2)}^q \right)^{1/q}.
\end{equation}
One can prove that the space  $S_{p,q}^{s_1,s_2}B(\mathbb{R}^2)$ does not depend on the chosen system $\varphi \in \Phi(\mathbb{R})$, see  \cite[Proposition 2.3.2/1]{Triebel1983B}, and the norms are pairwise equivalent.
It is known that $S_{2,2}^{s,\delta}B(\mathbb{R}^2) = H_t^sH_x^\delta (\mathbb{R}^2)$, for $s,\delta \in \mathbb{R}$, with   equivalent norms.

The next proposition justifies the coordinate transformation \eqref{CoordTranformation} from the computation perspective, since in $(\alpha,\beta)$-coordinate the knowledge of product Sobolev spaces is enough to have the local theory.
\begin{Proposition}\label{prop-HsDeltaPdtSobIsom}
	If  $s \geq \delta \in \mathbb{R}$, then the  map 
\begin{equation}\label{eqn-3.1}
		H^{s,\delta} \ni u(t,x) \mapsto u^*(\alpha,\beta) \in H_\alpha^s H_\beta^\delta \cap H_\beta^s H_\alpha^\delta =: \mathbb{H}^{s,\delta},
	\end{equation}
	is an isomorphism,   where as usual the space $\mathbb{H}^{s,\delta}$ is equipped with the norm
	\begin{equation}\nonumber
		\Vert u^* \Vert_{\mathbb{H}^{s,\delta}} := \sqrt{ \Vert u^* \Vert_{H_\alpha^s H_\beta^\delta}^2 + \Vert u^* \Vert_{H_\beta^s H_\alpha^\delta}^2}.
	\end{equation}
	In particular, we have \begin{equation}\nonumber
		\Vert u^* \Vert_{\mathbb{H}^{s,\delta}} \lesssim \Vert u \Vert_{H^{s,\delta} } \lesssim \Vert u^* \Vert_{\mathbb{H}^{s,\delta}}.
	\end{equation}
\end{Proposition}

With the spaces $\mathbb{H}^{H_1,H_2}$  defined above introduced  in the previous Proposition \ref{prop-HsDeltaPdtSobIsom}, we  have the following result.
\begin{Proposition}\label{prop-NoiseWelDef}
	Assume that  $H_1,H_2 \in \bigl(0, 1 \bigr)$ and  $H_i^\prime \in (0, H_1 \wedge H_2)$, $i=1,2$.
	Then  there exists a complete filtered probability space $(\Omega,\mathfrak{F},\mathbb{P})$ and a map,
	\[ \varXi: \mathbb{R}_+^2 \times \Omega \to \mathbb{R}, \]
	such that $\mathbb{P}$-a.s.  $\varXi(\cdot,\cdot,\omega) \in \mathbb{H}^{H_1^\prime,H_2^\prime}$ locally, i.e. for every bump function $\eta$,
	\begin{equation}\nonumber
	\eta(\alpha) \eta(\beta) \varXi(\alpha,\beta,\omega) \in \mathbb{H}^{H_1^\prime,H_2^\prime}.
	\end{equation}
	Moreover, for  $(\alpha_1,\beta_1), (\alpha_2,\beta_2) \in \mathbb{R}^2$,
	\begin{equation}\nonumber
	\mathbb{E}\left[\varXi(\alpha_1, \beta_1) ~ \varXi(\alpha_2, \beta_2)\right] = R_{H_1}(\vert \alpha_1 \vert,\vert \alpha_2 \vert) ~ R_{H_2}(\vert \beta_1 \vert,\vert \beta_2 \vert).
	\end{equation}
	Here $\mathbb{E}$ is the Expectation operator w.r.t. $\mathbb{P}$.	
\end{Proposition}
The above result, somehow related to a result proved in \cite{QT2007}, can be proved by using Proposition \ref{prop-HsDeltaPdtSobIsom}  and a combination of results from \cite{CC1979} and \cite{FLP1999}.

\section{The main result: the local well-posedness}
Let us fix $s \geq \delta \in \left( \frac{3}{4},1 \right)$ for  the whole present  section.
To  solve the SGWE problem \eqref{eqn-SGWEIntegralForm} locally,  which is sufficient to prove the local-wellposedness result we are aiming, let $\eta, \chi \in \mathcal{C}_{0}^\infty(\mathbb{R};[0,1])$ be even  cut-off functions   such that $\supp\,\eta=\supp\,\chi \subset [-4,4]$ and $[-2,2] \subset \eta^{-1}(\{1\})=\chi^{-1}(\{1\})$.
Let us put  $\eta_T(x) := \eta(x/T)$, $x\in \mathbb{R}$,  for any $T>0$. Similarly, we define  $\chi_T$.

To simplify the exposition, without loss of generality, we restrict ourselves to the target manifold of dimension $2$ and which can be covered by a family of charts such that the Christoffel symbols  $\Gamma_{ab}^k$ depend polynomially on $u$, that is, for every $k=1,2$,  one can find $r \in \mathbb{N}$ and $A_{ab}^l \in \mathbb{R}^2$ such that
$\Gamma_{ab}^k(u) = \sum_{\vert l \vert  \leq r} A_{ab}^l u^l ,\;\; u=(u^1,u^2) \in \mathbb{R}^2$,
where,  for  $l = (l_1,l_2) \in \mathbb{N}^2$,
$u^l=[u^1]^{l_1} [u^2 ]^{l_2}$.

Our  first result in  this section is a generalization of \cite[Lemma 2.2]{PT2016}.
\begin{Proposition}\label{StocIntWelDef}
		Assume that $s, \delta \in \left( \frac{3}{4},1\right)$, such that $s \geq \delta$, and  $f \in \mathbb{H}^{s-1,\delta-1}$ are given.  Then $F: \mathbb{R}^2 \to \mathbb{R}^2$ defined by $F:= H + I + J + G$ where, for $\alpha,\beta \in \mathbb{R}$,
		\begin{equation}\nonumber
			H(\alpha,\beta) := \int_{-\beta}^{\alpha} \int_{-\gamma}^{\beta} (\Delta_{0,0}f)(\gamma,\tau)\, d\tau\, d\gamma,
		\end{equation}
		\begin{align}
		I(\alpha,\beta) & := \int_{-\beta}^{\alpha} \sum\limits_{n = 1}^\infty \mathcal{F}^{-1} \left[ \frac{1}{i\xi}\varphi_0(\tau) \varphi_n(\xi) (\mathcal{F}f)(\tau,\xi) \right](\gamma,\beta) \, d\gamma, \nonumber\\
		&  \quad -  \int_{-\beta}^{\alpha} \sum\limits_{n = 1}^\infty \mathcal{F}^{-1} \left[ \frac{1}{i\xi}\varphi_0(\tau) \varphi_n(\xi) (\mathcal{F}f)(\tau,\xi) \right](\gamma,-\gamma) \, d\gamma, \nonumber
		\end{align}
		\begin{align}
		J (\alpha,\beta) & := \int_{-\alpha}^{\beta} \sum\limits_{m = 1}^\infty \mathcal{F}^{-1} \left[ \frac{1}{i\tau}\varphi_0(\xi) \varphi_m(\tau) (\mathcal{F}f)(\tau,\xi) \right](\alpha,\gamma) \, d\gamma \nonumber\\
		& \quad - \int_{-\alpha}^{\beta} \sum\limits_{m = 1}^\infty \mathcal{F}^{-1} \left[ \frac{1}{i\tau}\varphi_0(\xi) \varphi_m(\tau) (\mathcal{F}f)(\tau,\xi) \right](-\gamma,\gamma) \, d\gamma, \nonumber
		\end{align}
		and
		\begin{align}
		G(\alpha,\beta) & := \sum_{j,k=1}^{\infty}  \left[ \mathcal{F}^{-1} [ \frac{1}{(i \tau)(i\xi)} \varphi_j(\tau) \varphi_k(\xi)  (\mathcal{F} \phi)(\tau,\xi)  ]\right](\alpha,\beta) \nonumber\\
		& \quad - \frac{1}{2} \sum_{j,k=1}^{\infty}  \left[ \mathcal{F}^{-1} [ \frac{1}{(i \tau)(i\xi)} \varphi_j(\tau) \varphi_k(\xi)  (\mathcal{F} \phi)(\tau,\xi)  ]\right](\alpha,-\alpha) \nonumber\\
		& \quad - \frac{1}{2} \sum_{j,k=1}^{\infty}  \left[ \mathcal{F}^{-1} [ \frac{1}{(i \tau)(i\xi)} \varphi_j(\tau) \varphi_k(\xi)  (\mathcal{F} \phi)(\tau,\xi)  ]\right](-\beta,\beta) \nonumber\\
		& \quad - \frac{1}{2} \int_{-\beta}^{\alpha} \sum_{j,k=1}^{\infty}  \left[ \mathcal{F}^{-1} [ \frac{1}{(i\xi)} \varphi_j(\tau) \varphi_k(\xi)  (\mathcal{F} \phi)(\tau,\xi)  ]\right](\gamma,-\gamma) \, d\gamma \nonumber\\
		& \quad - \frac{1}{2} \int_{-\beta}^{\alpha} \sum_{j,k=1}^{\infty}  \left[ \mathcal{F}^{-1} [ \frac{1}{(i\tau)} \varphi_j(\tau) \varphi_k(\xi)  (\mathcal{F} \phi)(\tau,\xi)  ]\right](\gamma,-\gamma) \, d \gamma, \nonumber
		\end{align}
		 is the unique tempered distribution such that $\frac{\partial^2 F}{\partial\alpha \partial \beta} =f$, and satisfy the following homogeneous boundary conditions
		 \begin{align}
		 F(\alpha,-\alpha) =0 \quad \textrm{ and } \quad \frac{\partial F}{\partial\alpha } (\alpha,-\alpha) + \frac{\partial F}{\partial\beta } (\alpha,-\alpha)  =0. \nonumber
		 \end{align}
		 Moreover, for every $\eta,\chi$ and $T >0$, there exists $C(\eta,\chi,T) >0$ such that
		\begin{equation}\nonumber
		\Vert \eta_T(\alpha) \chi_T(\beta) F(\alpha,\beta) \Vert_{\mathbb{H}^{s,\delta}} \leq C(\eta,\chi, T)~  \Vert f \Vert_{\mathbb{H}^{s-1,\delta-1}}.
		\end{equation}
		We will use the following notation
		\begin{equation}\nonumber
		F(\alpha,\beta) =: \int_{-\beta}^{\alpha} \int_{-a}^{\beta} f(da,db), \;\; (\alpha,\beta) \in \mathbb{R}^2.
		\end{equation}
\end{Proposition}
\begin{proof}
Using the properties of $S_{2,2}^{s,\delta}B(\mathbb{R}^2)$ spaces, we need to show that $G,H, I, J$ are well-defined elements of $\mathbb{H}^{s,\delta}$.
\end{proof}

By following the approach of \cite{GT2010} we get the next required result.
\begin{Proposition}\label{DiffusionTerm}
	Assume that  $\sigma \in \mathcal{C}_b^3(\mathbb{R}^2)$.  Then $\sigma \circ u \in \mathbb{H}^{s,\delta}$ for every
 $u \in \mathbb{H}^{s,\delta}$  and there exist constants  $C_i(\sigma):=C_i (\Vert \sigma \Vert_{\mathcal{C}_b^{i+1}})$, $i=1,2$ such that for
 $u,u_1,u_2 \in \mathbb{H}^{s,\delta}$,
\begin{equation}\nonumber
	\Vert \sigma \circ u \Vert_{\mathbb{H}^{s,\delta}}^2 \leq C_1(\sigma)  \Vert u \Vert_{\mathbb{H}^{s,\delta}}^2 \left[ 1 + \Vert u \Vert_{\mathbb{H}^{s,\delta}}^2 \right],
	\end{equation}
\begin{equation}\nonumber
\begin{split}
\Vert \sigma \circ u_1 - \sigma \circ u_2 & \Vert_{\mathbb{H}^{s,\delta}}^2 \leq C_2(\sigma) \Vert u_2-u_1 \Vert_{\mathbb{H}^{s,\delta}}^2 \left[ 1 + \sum_{i,k=1}^{2}    \Vert u_i \Vert_{\mathbb{H}^{s,\delta}}^{2k}  \right]. \nonumber
\end{split}	
\end{equation}
\end{Proposition}
We now state and provide a sketch of proof of the main result of this note. Below we fix a realisation of the random field belonging to the space $\mathbb{H}^{s,\delta}$, see Prop. \ref{prop-NoiseWelDef}.
\begin{Theorem}\label{thm-SLWP}
	Assume $s,\delta \in \left( \frac{3}{4},1 \right)$ such that $\delta \leq s$  and $(u_0,u_1) \in H^s(\mathbb{R}) \times H^{s-1}(\mathbb{R})$. Let $\varXi$ be a fractional Brownian sheet with Hurst indices $H_1, H_2 \in (s,1)$.
	There exist a $R_0 \in (0,1)$ and a $ \lambda_0 := \lambda_0(\| u_0\|_{H^s}, \| u_1\|_{H^{s-1}}, R_0)  \gg 1$  such that for every $\lambda \geq \lambda_0$ there exists a unique $u := u(\lambda,R_0) \in \mathbb{B}_{R_0}$, where $\mathbb{B}_R := \left\{u \in \mathbb{H}^{s,\delta}: \| u\|_{\mathbb{H}^{s,\delta}} \leq R \right\},$ which satisfies the following integral equation
	\begin{align}
	u(\alpha,\beta)  = \eta(\lambda \alpha)  \eta(\lambda \beta) & \left( [S( \chi(\lambda) u_0  ,\chi(\lambda) u_1)](\alpha,\beta)  + [\Diamond^{-1} \mathcal{N}(u)](\alpha,\beta)  \right. \nonumber\\
	& \quad \qquad\left. + [\Diamond^{-1} \sigma(u) \varXi_{\alpha\beta}](\alpha,\beta) \right), \qquad (\alpha,\beta) \in \mathbb{R}^2. \nonumber
	\end{align}
	Here the right hand side terms are, respectively, defined in  \eqref{SGWEIntegralFormHomgeneous}, \eqref{SGWEIntegralFormInHomgeneous} and \eqref{SGWEIntegralFormNoise}.

\end{Theorem}
\begin{proof}[\textbf{Sketch of proof of Theorem \ref{thm-SLWP}}]
	Our proof  is based on the Banach Fixed Point Theorem in the space $\mathbb{H}^{s,\delta}$. Note that all the constants below are positive and depend on $\eta$ unless mentioned otherwise.
	
	\textbf{Step 1:} Using the following well-known result, see e.g.  \cite{G1996},
 \begin{equation}\nonumber
		\bigg\Vert \bigl\{ x \mapsto  \chi(x) \int_{0}^{x}f(y) \, dy \bigr\}\bigg\Vert_{H^s} \lesssim \Vert f \Vert_{H^{s-1}},
	\end{equation}
	we can estimate the localized homogeneous part of the solution  as
	\begin{align}\nonumber
		\Vert \eta(\alpha) \chi(\beta) S(u_0,u_1) \Vert_{\mathbb{H}^{s,\delta}} \leq C_S\left( \Vert u_0 \Vert_{H^s} + \Vert u_1 \Vert_{H^{s-1}}\right).
	\end{align}

	\textbf{Step 2:} In view of the polynomiality of the Christoffel symbols, by using Prop. \ref{StocIntWelDef},  we deduce the existence of a natural number $\gamma \geq 2$ such that
	\begin{equation}\nonumber
		\Vert \eta(\alpha) \chi(\beta) ~ \Diamond^{-1}(\mathcal{N}(\phi) - \mathcal{N}(\psi)) \Vert_{\mathbb{H}^{s,\delta}} \leq C_{\mathcal{N}} \Vert \phi-\psi \Vert_{\mathbb{H}^{s,\delta}} \left[ \Vert \phi \Vert_{\mathbb{H}^{s,\delta}} + \Vert \psi \Vert_{\mathbb{H}^{s,\delta}}\right]^{\gamma}.
	\end{equation}
	
	\textbf{Step 3:} By  Propositions \ref{StocIntWelDef} and \ref{DiffusionTerm} followed by the continuity of the multiplication map
	$$\mathbb{H}^{s,\delta} \times  \mathbb{H}^{s-1,\delta-1} \to  \mathbb{H}^{s-1,\delta-1},$$
	see e.g. \cite{KT1998},  we get
	\begin{align}
	\Vert \eta(\alpha) \eta(\beta) &  \Diamond^{-1} [(\sigma (u_1) - \sigma (u_2)) \varXi_{\alpha\beta}] \Vert_{\mathbb{H}^{s,\delta}}^2 	\nonumber\\		
	& \quad \quad \leq C_{\varXi} ~ C_2(\sigma) ~ \Vert u_2-u_1 \Vert_{\mathbb{H}^{s,\delta}} \left[ 1 + \sum_{i,k=1}^{2}    \Vert u_i \Vert_{\mathbb{H}^{s,\delta}}^{k}  \right] ~ \Vert \varXi_{\alpha\beta} \Vert_{\mathbb{H}^{s-1,\delta-1}}, \nonumber
	\end{align}
	for some positive constants $C_{\varXi}$ and $C_2(\sigma):=C_2 (\Vert \sigma \Vert_{\mathcal{C}_b^{3}})$.

	\textbf{Step 4:} We consider a map $\Theta^\lambda: \mathbb{H}^{s,\delta} \ni u \mapsto u_{\Theta^\lambda} \in \mathbb{H}^{s,\delta}$ defined by
	\begin{align}
	u_{\Theta^\lambda} = \eta(\alpha) \eta(\beta) \left( S( u_0^\lambda, u_1^\lambda ) + \Diamond^{-1} [\mathcal{N}(u)]+ \Diamond^{-1}[ \sigma(u) \varXi_{\alpha\beta}^\lambda]  \right),	\nonumber
	\end{align}
	where
	\begin{align}
	u_0^{\lambda}(\alpha) := \chi(\alpha) \left[ u_0\left( \frac{\alpha}{\lambda}\right) - \bar{u}_0^{\lambda} \right] , \,\,	u_1^\lambda(\alpha) := \chi(\alpha)  \lambda^{-1} u_1\left( \frac{\alpha}{\lambda}\right), \,\, \textrm{ and }\,\, \varXi_{\alpha\beta}^\lambda := \lambda^{-2} \Pi_{\lambda} \varXi_{\alpha\beta}, \nonumber
	\end{align}
	with
	\begin{equation}\nonumber
	\bar{u}_0^{\lambda} := \int_{\mathbb{R}} u_0 \left( \frac{y}{\lambda}\right) \psi(y)\,dy .
	\end{equation}
	Here $\psi$ is any bump function which is non zero on the support of $\chi, \eta$ and $\int_{\mathbb{R}}\psi(x)\, dx =1$.
	Then, in view of the assumption that $s+ \delta > \frac{3}{2}$, by using the above estimates we obtain, for any $u,v \in \mathbb{B}_R$,
	\begin{align}
	\| u_{\Theta^\lambda} - v_{\Theta^\lambda}\|_{\mathbb{H}^{s,\delta}} & \lesssim   \left[ C_{\mathcal{N}}    R  + \lambda^{1-(s+\delta)} C_\varXi ~ C_2(\sigma) ~ \left( 1 + R \right) ~ \Vert \varXi_{\alpha\beta} \Vert_{\mathbb{H}^{s-1,\delta-1}}  \right] ~ \Vert u - v \Vert_{\mathbb{H}^{s,\delta}}. \nonumber
	\end{align}
	Hence we can choose $R_0 \in (0,1), \lambda_0: = \lambda_0(\| u_0\|_{H^s}, \| u_1\|_{H^{s-1}}, R_0)$ in such a way that $\Theta^\lambda$ is $\frac{1}{2}$-contraction as a map from  $\mathbb{B}_{R_0}$ into itself and, then by the Banach Fixed Point Theorem there exists a unique $u^\lambda \in \mathbb{B}_{R_0}$ such that $u^\lambda = \Theta^\lambda(u^\lambda)$.

	\textbf{Step 5:} By working with another suitable translated coordinate chart on $\mathcal{N}$ (which will remove the dependence on $\bar{u}_0^{\lambda}$) and by defining the inverse scaling $u(\alpha,\beta) :=  u^\lambda(\lambda \alpha,\lambda\beta)$ for the  fixed point
	$u^\lambda$ from Step 4, we deduce that
	\begin{align}
	u(\alpha,\beta)  & = \Theta^\lambda\left( u^\lambda\right) (\lambda \alpha,\lambda \beta)  \nonumber \\
	& = \eta(\lambda \alpha)  \eta(\lambda \beta) \left( [S( \chi(\lambda) u_0  , \chi(\lambda)u_1)] + [\Diamond^{-1} \mathcal{N}(u)]  + [\Diamond^{-1} \sigma(u) \varXi_{\alpha\beta}] \right). \nonumber
	\end{align}
	Hence we conclude the proof of Theorem \ref{thm-SLWP}.		
\end{proof}
We complete our study of local theory with the following Theorem.
\begin{Theorem}\label{thm-SLWP2}
	Under the above mentioned assumptions, there exist a open set $\mathcal{O}$, containing the diagonal $\mathcal{D} := \{ (\alpha,-\alpha): \alpha \in \mathbb{R} \}$, and  a function  $u: \mathcal{O} \to \mathbb{R}^2$ such that for every $(\alpha_0,-\alpha_0) \in \mathcal{D}$, there exists $r>0$ such that $\restr{u}{B_r((\alpha_0,-\alpha_0))} \in \mathbb{H}^{s,\delta}$, where $B_r((\alpha,-\alpha))$ is open ball of radius $r$ around $(\alpha,-\alpha)$, and $u$ solves \eqref{SGWEab}-\eqref{SGWEabCons} uniquely in $\mathcal{O}$.
\end{Theorem}
To prove Theorem \ref{thm-SLWP2}, for each fixed point $(\alpha_0,-\alpha_0) \in \mathcal{D}$, by Theorem \ref{thm-SLWP}  we  find a unique solution $u_{\alpha_0}$ of  a translated version of  the  problem \eqref{eqn-SGWEIntegralForm} defined in some neighbourhood $N_{\alpha_0}$ of $(\alpha_0,-\alpha_0)$. By using the uniqueness  we can glue ``local" solutions to get a solution $u$ as in the assertion.\\

\textbf{Acknowledgements:} We would like to thank Prof. Kenji Nakanishi for many helpful conversations via email related to linear estimates from his paper \cite{MNT2010} and to provide the reference \cite{G1996}. The second author also wishes to thank   the Department of Mathematics, University of York, and the ORS, for  financial support and excellent research facilities.

\end{document}